\documentclass[12pt, twoside, fleqn, a4paper]{article}

\usepackage{latexsym}  
\usepackage{amscd}    
\usepackage{theorem}  
\usepackage{pifont}  
\usepackage{mathbbol} 
\usepackage{amsfonts}  
\usepackage{xspace}  
\usepackage{amssymb} 
\usepackage{fancyhdr} 
\usepackage{mathrsfs} 
\usepackage{amsmath} 
\usepackage{graphics} 
\usepackage{graphicx} 
\usepackage{euscript}
\usepackage{psfrag}
\usepackage{empheq} 
\usepackage{mathabx} 
\usepackage[parfill]{parskip}     
\usepackage[colorlinks=true, allcolors=blue]{hyperref}
\usepackage{cancel} 

\title{Counting functions over periodic orbits of a skew-product map}
\author{Subith Gopinathan\footnote{Indian Institute of Science Education and Research Thiruvananthapuram (IISER-TVM). \\ email: \texttt{subith21@iisertvm.ac.in}}, \\ Bharath Krishna Seshadri\footnote{Indian Institute of Science Education and Research Thiruvananthapuram (IISER-TVM). \\ email: \texttt{bharathmaths21@iisertvm.ac.in}} and \\ Shrihari Sridharan\footnote{Corresponding Author. Indian Institute of Science Education and Research Thiruvananthapuram (IISER-TVM). email: \texttt{shrihari@iisertvm.ac.in}}}

\DeclareFontFamily{OT1}{pzc}{}
\DeclareFontShape{OT1}{pzc}{m}{it}%
              {<-> s * [0.900] pzcmi7t}{}
\DeclareMathAlphabet{\mathpzc}{OT1}{pzc}%
                                 {m}{it}

{\theorembodyfont{\slshape} \newtheorem{theorem}{Theorem}[section]}
{\theorembodyfont{\slshape} \newtheorem{definition}[theorem]{Definition}}
{\theorembodyfont{\slshape} }
{\theorembodyfont{\slshape} }
{\theorembodyfont{\slshape} \newtheorem{proposition}[theorem]{Proposition}}
{\theorembodyfont{\slshape} }
{\theorembodyfont{\slshape} \newtheorem{corollary}[theorem]{Corollary}} 
{\theorembodyfont{\slshape} } 
{\theorembodyfont{\slshape} }
{\theorembodyfont{\slshape} }

\numberwithin{equation}{section}

\newenvironment{proof}{\paragraph{Proof:}}{\hfill$\bullet$}

\begin{document}

\maketitle

\begin{abstract} 
\noindent 
In this manuscript, we investigate some properties of certain counting functions, associated to the ergodic sums computed along the periodic orbits of the skew-product map, related to a finitely generated rational semigroup. To be precise, we obtain some comparability results for the above mentioned counting functions. 
\end{abstract}

\begin{tabular}{l l} 
{\bf Keywords} & Skew-product map \\ 
& Orbit counting functions \\ 
& Comparable sequences \\ 
& \\ 
{\bf MSC Subject} & 37A44, 37F10, 37F20, 11N45. \\ 
{\bf Classifications} & \\ 
\end{tabular} 
\bigskip 

\newpage 

\section{Introduction} 

The theory of dynamical systems is concerned with the study of iterations of a given map $T$ defined from a set $X$ to itself. One may ascribe various structures to $X$, say topological, measurable, analytic {\it etc.}, and investigate interesting properties of the iterates of $T$ in each context, with $T$ possessing appropriate properties that preserves the given structure. For a given map $T$, the points which cycle back to themselves after a finite sequence of iterations of $T$, play a key role in the dynamics of $T$. These are called the periodic points of $T$ and a closely related notion is a closed orbit of finite length. 

There are typically two ways to study periodic points. One is to use a type of generating function called the dynamical zeta function, the study of which has its origins in the work of Artin and Mazur in \cite{am:1965} and gained prominence due to the works of Ruelle, Parry, Pollicott, Sharp, Bowen and Lanford, among others, see for example \cite{wp:1983, pp:1983, pp:1990, rs:1991, dr:2002}. This study has also been inspired by the Riemann zeta function from number theory. Two of the classic references for a survey of results on this topic are \cite{pp:1990, dr:2002}.

The second approach is to define various counting functions arising as a result of studying the number of periodic points for a given period and finding various asymptotes to detect the order of growth of these counting functions. The prime number theorem is one of the important results in number theory and it gives an asymptote for the number of primes less than or equal to a given $x \in \mathbb{R}^{+}$. The theorem, proved independently by Hadamard and de la Vallée Poussin, says that the prime counting function $\pi(x)= \displaystyle \sum_{p \leq x} 1$ (where $p$ is a prime) is asymptotic to the function $\dfrac{x}{\log x}$. To be precise, it says $\displaystyle \lim_{x \to \infty} \dfrac{\pi(x)}{\frac{x}{\log x}} = 1$. Our prime focus in this work is to define such counting functions for a randomly generated dynamical system, from a finite set of maps, that we shall soon define. This approach has its origins in the works of Parry and Pollicott, see for example \cite{wp:1983, pp:1983}, although these works contain some other vast generalisations of the above concepts as well. 

The techniques we use in this paper have been studied in \cite{emsw:2007} with the focus being on non-hyperbolic dynamical systems and later in the works \cite{ana:2015, ana:2017}, which focus on the Dyck shift and the Motzkin shift, respectively. A survey of these works can be found in \cite{nnd:2019}. In this paper, we work with the skew product map related to a given rational semigroup, the definitions of which shall be introduced in Section \eqref{sec:prelim}. The counting functions we have introduced here are different from the ones defined in \cite{ana:2015, ana:2017}, as it shall be clear to the reader from the definitions given in Section \eqref{sec:prelim}. We have thus generalised the results in \cite{ana:2015, ana:2017} for the skew product map using our new counting functions, which also include some information about the ergodic sum. We also give a few details about the number theoretic and dynamical analogues of our theorems in various contexts, in between the proofs of our results. 

The skew product map that we deal with in this work was introduced by Sumi in \cite{sm1:2000} in order to study a rational semigroup. Over time, it has become an interesting dynamical system in its own right and has been a focal point of various results with regards to topological and measure theoretic aspects - see for example, \cite{sm2:2001, su:2009, stt:2021}. 


\section{Preliminaries and the main results} 
\label{sec:prelim}

Consider finitely many rational maps, say $R_{1}, R_{2}, \cdots, R_{M}$ acting on the Riemann sphere, $\widehat{\mathbb{C}} = \mathbb{C} \cup \{ \infty\}$, having degrees $r_{1}, r_{2}, \cdots, r_{M}$, respectively. We assume that there exists at least one rational map in the considered collection whose degree is at least $2$. For such a collection, we define a skew-product map $S : X = \Sigma_{M}^{+} \times \widehat{\mathbb{C}} \righttoleftarrow$ defined by $S \left( \omega, z_{0} \right) = \left( \sigma \omega, R_{\omega_{1}} z_{0} \right)$ where $\omega = \left( \omega_{1} \omega_{2} \cdots \right)$ is some infinitely long word with letters drawn from $\left\{ 1, 2, \cdots, M \right\},\ \sigma : \Sigma_{M}^{+} \righttoleftarrow$ is the shift map defined by $\left( \sigma \omega \right)_{j} = \omega_{j + 1}$ for all $j \ge 1$. By virtue of this definition, the $n^{{\rm th}}$ iterate of $S$ is given by $S^{n} \left( \omega, z_{0} \right) = \left( \sigma^{n} \omega, \left( R_{\omega_{n}} \circ \cdots \circ R_{\omega_{1}} \right) z_{0} \right)$. We accord the product topology on the space $X$, borne out of the product metric between $d_{\Sigma_{M}^{+}}$ on $\Sigma_{M}^{+}$ and the spherical metric $d_{\widehat{\mathbb{C}}}$ on the Riemann sphere, where 
\[ d_{\Sigma_{M}^{+}} \left( \omega, \eta \right)\ =\ \begin{cases} \dfrac{1}{2^{n(\omega, \eta)}} & \text{for}\ \ \omega \ne \eta\ \ \text{where}\ \ n(\omega, \eta) = \min \left\{ n \in \mathbb{Z}_{+} : \omega_{n} \ne \eta_{n} \right\} \vspace{+5pt} \\ 
0 & \text{if}\ \ \ \omega = \eta. \end{cases} \]

\begin{definition} 
A point $(\omega, z_{0}) \in X$ is called a periodic point of $S$ of period $n \in \mathbb{Z}_{+}$ if $S^{n} (\omega, z_{0}) = (\omega, z_{0})$; meaning $\omega \in \Sigma_{M}^{+}$ is some $n$-lettered word concatenated to itself infinitely many times and $\left( R_{\omega_{n}} \circ \cdots \circ R_{\omega_{1}} \right) z_{0} = z_{0}$. Further, $n$ is said to be the prime period for the point $(\omega, z_{0})$ if $n$ is the least positive integer for which $S^{n} (\omega, z_{0}) = (\omega, z_{0})$. 
\end{definition} 

By ${\rm Per}_{n} (S)$ and ${\rm Per}'_{n} (S)$, we denote the set of all periodic points of period $n$ for $S$ and the set of all periodic points of prime period $n$ for $S$, respectively. Let $\mathscr{O}_{n} (S)$ denote the collection of all closed orbits of $S$ of length $n$ in $X$, where by a closed orbit of length $n$, we mean the set $\left\{ (\omega, z_{0}), S (\omega, z_{0}), \cdots, S^{n - 1} (\omega, z_{0}) \right\}$ that satisfies $S^{n} (\omega, z_{0}) = (\omega, z_{0})$. 

We consider the set of all functions $f : X \longrightarrow \mathbb{R}$ and denote the $n^{{\rm th}}$ order ergodic sum of $f$ as $f^{n}$ that is defined by 
\[ f^{n} (\omega, z_{0})\ \ =\ \ f (\omega, z_{0}) + f \left( S (\omega, z_{0}) \right) + \cdots + f \left( S^{n - 1} (\omega, z_{0}) \right). \] 
Thus, for any function $f$, its $n^{{\rm th}}$ order ergodic sum along any closed orbit of length $n$, say $\tau = \left\{ (\omega, z_{0}), S (\omega, z_{0}), \cdots, S^{n - 1} (\omega, z_{0}) \right\} \in \mathscr{O}_{n} (S)$, is denoted by $f^{n} (\tau)$ and is defined by 
\[ f^{n} (\tau)\ \ =\ \ f (\omega, z_{0}) + f \left( S (\omega, z_{0}) \right) + \cdots + f \left( S^{n - 1} (\omega, z_{0}) \right). \] 

Dynamicists are interested in a counting function, analogous to $\pi (x)$, as stated in the introduction, called the orbit counting function. In \cite{wp:1983}, Parry proved the prime orbit theorem for closed orbits of shifts of finite type and for suspension flows related to shifts. Further, in \cite{pp:1983}, Parry and Pollicott proved the prime orbit theorem for closed orbits of Axiom A flows. These works use the Weiner-Ikehara theorem that is used to prove the classical prime number theorem in number theory. More details on this can be found in \cite{pp:1990}.

We now define a few quantities, as follows: 
\begin{eqnarray*} 
E_{S} (f, n) & = & \sum_{p\, =\, (\omega, z_{0})\, \in\, {\rm Per}_{n} (S)} e^{f^{n_{p}} (\omega, z_{0})}\ \ \text{where}\ \ \ n_{p}\ \text{is the prime period of}\ (\omega, z_{0}) \\ 
D_{S} (f, n) & = & \sum_{(\omega, z_{0})\, \in\, {\rm Per}'_{n} (S)} e^{f^{n} (\omega, z_{0})} 
\end{eqnarray*} 
\begin{eqnarray*} 
C_{S} (f, n) & = & \sum_{\tau\, \in\, \mathscr{O}_{n} (S)} e^{f^{n} (\tau)}. 
\end{eqnarray*} 

Using the above quantities, we define the following counting function, called the \emph{prime orbit counting function}, given by 
\[ \pi_{S} (f, N)\ \ =\ \ \sum_{n\, \le\, N} C_{S} (f, n). \] 

We are now ready to state the main theorems of this paper, after we introduce the following notation. We say that a sequence $A_{n}$ is comparable to a sequence $B_{n}$, denoted as $A_{n} \asymp B_{n}$ if there exists positive constants $\kappa_{1}$ and $\kappa_{2}$ such that $\kappa_{1} B_{n} \le A_{n} \le \kappa_{2} B_{n}$ for all $n \in \mathbb{Z}_{+}$, meaning that the two sequences have the same order of growth. 

\begin{theorem} 
\label{thm:one}
Consider the skew-product map $S$ defined on $X$. Let $f : X \longrightarrow \mathbb{R}$. Suppose there exist a constant $\lambda_{f} > 1$ such that $E_{S} (f, n) \asymp \left( \lambda_{f} \right)^{n}$. Then, for $N \in \mathbb{Z}_{+}$, we have 
\[ \pi_{S} (f, N)\ \ =\ \ \sum_{n\, \le\, N} C_{S} (f, n)\ \ \asymp\ \ \frac{\left( \lambda_{f} \right)^{N}}{N}. \]
\end{theorem} 

\begin{theorem} 
\label{thm:two}
Consider the skew-product map $S$ defined on $X$. Let $f : X \longrightarrow \mathbb{R}$ satisfy the hypothesis, as stated in Theorem \ref{thm:one}. Then, for $N \in \mathbb{Z}_{+}$, we have  
\begin{equation} 
\sum_{n\, \le\, N}  \frac{C_{S} (f, n)}{\lambda_{f}^{n}}\ \ \asymp\ \ \log N. 
\end{equation} 
\end{theorem} 

\begin{theorem} 
\label{thm:three}
Consider the skew-product map $S$ defined on $X$. Let $f : X \longrightarrow \mathbb{R}$ satisfy the hypothesis, as stated in Theorem \ref{thm:one}. Then, for any fixed $k > 0$, we have 
\begin{equation} 
\sum_{n\, \ge\, 1}  \frac{C_{S} (f, n)}{n^{k} \lambda_{f}^{n}}\ \ \asymp\ \ \dfrac{1}{k}. 
\end{equation} 
\end{theorem} 

\begin{theorem} 
\label{thm:four}
Consider the skew-product map $S$ defined on $X$. Let $f : X \longrightarrow \mathbb{R}$ satisfy the hypothesis, as stated in Theorem \ref{thm:one}. Then, 
\begin{equation} 
\sum_{n\, \ge\, 1} \frac{C_{S} (f, n)}{n^{z}}\ \ \asymp\ \ \frac{1}{\zeta (z)} \sum_{n\, \ge\, 1} \frac{\lambda_{f}^{n}}{n^{z + 1}}, 
\end{equation} 
where $\zeta$ is the Riemann zeta function and ${\rm Re} (z) > 1$. 
\end{theorem} 

\section{Proof of Theorem \ref{thm:one}} 
\label{sec:pot}

In \cite{wp:1983}, the orbit counting function for sub-shifts of finite type is defined as $\pi(x) = \displaystyle \sum_{\tau\, :\, |\tau| \leq x} 1$, where $|\tau|$ denotes the length of the closed orbit $\tau$ of the shift map. The notion of asymptotes considered in \cite{wp:1983} is stronger than the notion of comparability that we consider, even as we work with a different dynamical system altogether. In \cite{ana:2015, ana:2017}, the authors focus on comparability results for the prime orbit counting function $\pi_{T}(N) = \displaystyle\sum_{n\, \leq\, N} O_{T} (n)$, where $O_{T} (n)$ is the number of orbits of the map $T$ of length $n$ and $T$ denotes the Dyck shift and Motzkin shift respectively. In this work, we have defined a generalised version of the orbit counting function given by $\pi_{S} (f,N)$ for the skew product map related to a rational semigroup and study comparability results related to it.

Now, we proceed to prove Theorem \ref{thm:one} using Propositions \ref{prop:pisfnlb} and \ref{prop:pisfnub}, that gives a lower and an upper bound for $\pi_{S} (f, n)$, respectively. From the hypothesis of Theorem \ref{thm:one}, there exist constants $\kappa_{1} > 0$ and $\kappa_{2} > 0$ such that $\kappa_{1} \left( \lambda_{f} \right)^{n} \le E_{S} (f, n) \le \kappa_{2} \left( \lambda_{f} \right)^{n}$, for all $n \in \mathbb{Z}_{+}$. Consider 
\begin{equation} 
\label{ineq:ubesfn}
\sum_{n\, \le\, N} E_{S} (f, n)\ \le\ \frac{\kappa_{2}}{\kappa_{1}} E_{S} (f, N) \sum_{n\, \le\, N} \frac{1}{\lambda_{f}^{n}}\ \le\ c_{1} E_{S} (f, N)\ \ \text{where}\ \ c_{1} = \frac{\kappa_{2}}{\kappa_{1}} \frac{\lambda_{f}}{\lambda_{f} - 1}. 
\end{equation} 


Also, by a similar computation, 
\begin{equation} 
\label{ineq:lbesfn} 
\sum_{n\, \le\, N} E_{S} (f, n)\ \ \ge\ \ \frac{\kappa_{1}}{\kappa_{2}} E_{S} (f, N). 
\end{equation} 




From the definitions of the counting functions, as in Section \ref{sec:prelim}, one may observe using the M\"{o}bius inversion formula that 
\begin{eqnarray} 
\label{eqn:mobinv}
C_{S} (f, n) & = & \frac{1}{n} \sum_{d\; :\; d|n} \mu \left( \frac{n}{d} \right) E_{S} (f, d),\ \ \ \text{where} \\ 
\mu (k) & = & \begin{cases} 1 & \text{if}\ k\ \text{is square-free and has an even number of prime factors,} \\ -1 & \text{if}\ k\ \text{is square-free and has an odd number of prime factors,} \\ 1 & \text{if}\ k = 1, \\ 0 & \text{otherwise}. \end{cases} \nonumber 
\end{eqnarray} 

We now state and prove a proposition that gives a lower bound for $\pi_{S} (f,n)$. 

\begin{proposition}
\label{prop:pisfnlb}
There exists a strictly positive constant, say $\sigma_{1}$ such that 
\[ \pi_{S} (f, N)\ \ \ge\ \ \sigma_{1} \dfrac{\lambda_{f}^{N}}{N}. \] 
\end{proposition}

\begin{proof} 
Making use of the definition of $\pi_{S} (f, N)$ and Equation \eqref{eqn:mobinv}, we get 
\[ \pi_{S} (f, N)\ =\ \sum_{n\, \le\, N} \frac{1}{n} \sum_{d\, :\, d|n} \mu \left( \frac{n}{d} \right) E_{S} (f, d). \]
Keeping in mind the notation that $\left[ x \right]$ represents the greatest integer less than or equal to $x$, we consider 
\begin{eqnarray} 
\label{ineq:mund}
\sum_{d\, :\, d|n} \mu \left( \frac{n}{d} \right) E_{S} (f, d) & = & \sum_{d\, <\, n\ :\ d|n} \mu \left( \frac{n}{d} \right) E_{S} (f, d) + \mu(1) E_{S} (f, n) \nonumber \\ 
& \ge & E_{S} (f, n) - \sum_{d\, \le\, \left[ \frac{n}{2} \right]} E_{S} (f, d) \nonumber \\ 
& = & \kappa_{1} \lambda_{f}^{n} \left( 1 - c_{1} \frac{\kappa_{2}}{\kappa_{1}} \frac{1}{\lambda_{f}^{\frac{n}{2}}} \right) \nonumber \\ 
& \ge & \frac{\kappa_{1}}{2} \lambda_{f}^{n},\ \hspace{+4cm} \forall n > N_{0}. 
\end{eqnarray} 

Thus, 
\[ \pi_{S} (f, N)\ \ge\ \sum_{n\, =\, 0}^{N_{0}} \frac{1}{n} \sum_{d\, :\, d|n} \mu \left( \frac{n}{d} \right) E_{S} (f, d) + \sum_{n\, >\, N_{0}} \frac{1}{n} \sum_{d\, :\, d|n} \mu \left( \frac{n}{d} \right) E_{S} (f, d)\ \ge\ \frac{\kappa_{1}}{2} \frac{\lambda_{f}^{N}}{N}. \]
Choosing $\sigma_{1} = \dfrac{\kappa_{1}}{2}$, completes the proof. 
\end{proof} 

The next proposition gives us an upper bound for $\pi_{S} (f, N)$, thereby completing the proof of Theorem \ref{thm:one}. 

\begin{proposition} 
\label{prop:pisfnub}
There exists a strictly positive constant, say $\sigma_{2}$ such that 
\[ \pi_{S} (f, N)\ \ \le\ \ \sigma_{2} \dfrac{\lambda_{f}^{N}}{N}. \] 
\end{proposition} 

\begin{proof} 
Since $\mu (\cdot)$ is atmost $1$, we obtain 
\begin{eqnarray} 
\label{ineq:pisfnub} 
\pi_{S} (f, n) & \le & \left[ \sum_{n\, \le\, N - \left[ N^{\frac{1}{6}} \right]} \frac{1}{n} \sum_{d\, <\, n\; :\; d|n} E_{S} (f, d) + \sum_{n\, \le\, N - \left[ N^{\frac{1}{6}} \right]} \frac{1}{n} E_{S} (f, n) \right] \nonumber \\ 
& & +\ \left[ \sum_{N - \left[ N^{\frac{1}{6}} \right]\, <\, n\, \le\, N} \frac{1}{n} \sum_{d\, <\, n\; :\; d|n} E_{S} (f, d)  + \sum_{N - \left[ N^{\frac{1}{6}} \right]\, <\, n\, \le\, N} \frac{1}{n} E_{S} (f, n) \right] \nonumber \\ 
& \le & c_{1} \sum_{n\, \le\, N} E_{S} \left( f, \left[\frac{n}{2}\right] \right) + \sum_{n\, \le\, N - \left[ N^{\frac{1}{6}} \right]} \frac{1}{n} E_{S} (f, n) + \sum_{N - \left[ N^{\frac{1}{6}} \right]\, <\, n\, \le\, N} \frac{1}{n} E_{S} (f, n), \nonumber \\ 
\end{eqnarray} 


where the first term above is obtained, by virtue of inequality \eqref{ineq:ubesfn}. We now obtain a bound on the third term in the right hand side of the inequality \eqref{ineq:pisfnub}. 
\begin{eqnarray} 
\label{ineq:tsfnub}
\sum_{N - \left[ N^{\frac{1}{6}} \right]\, <\, n\, \le\, N} \frac{1}{n} E_{S} (f, n) & = & \frac{E_{S} (f, N)}{N}\ \sum_{r\, =\, 0}^{\left[ N^{\frac{1}{6}} \right] - 1} \frac{N}{E_{S} (f, N)} \frac{E_{S} (f, N - r)}{N - r} \nonumber \\ 
& \le & \frac{\kappa_{2} \lambda_{f}^{N}}{N}\ \sum_{r\, =\, 0}^{\left[ N^{\frac{1}{6}} \right] - 1} \frac{\kappa_{2}}{\kappa_{1}} \frac{1}{\lambda_{f}^{r}} \sum_{n\, \ge\, 0} \left( \frac{r}{N} \right)^{n} \nonumber \\ 
& \le & \frac{\kappa_{2}^{2}}{\kappa_{1}} \frac{\lambda_{f}^{N}}{N} \left( \frac{\lambda_{f}}{\lambda_{f} - 1} + \sum_{r\, =\, 0}^{\left[ N^{\frac{1}{6}} \right] - 1} \frac{r}{N - r} \right) \nonumber \\ 
& \le & \frac{\kappa_{2}^{2}}{\kappa_{1}} \frac{\lambda_{f}^{N}}{N} \left( \frac{\lambda_{f}}{\lambda_{f} - 1} + \frac{\left[ N^{\frac{1}{3}} \right]}{N - \left[ N^{\frac{1}{6}} \right]} \right). 
\end{eqnarray} 

Thus, plugging in the upper bound for each of the terms in inequality \eqref{ineq:pisfnub} using the inequalities in \eqref{ineq:ubesfn} and \eqref{ineq:tsfnub}, we obtain 
\begin{eqnarray*} 
\pi_{S} (f, N) & \le & c_{1}^{2} E_{S} \left( f, \left[\frac{N}{2}\right] \right) + c_{1} E_{S} \left( f, N - \left[ N^{\frac{1}{6}} \right] \right) + \frac{\kappa_{2}^{2}}{\kappa_{1}} \frac{\lambda_{f}^{N}}{N} \left( \frac{\lambda_{f}}{\lambda_{f} - 1} + \frac{\left[ N^{\frac{1}{3}} \right]}{N - \left[ N^{\frac{1}{6}} \right]} \right) \\ 
& \le & \kappa_{3} \frac{\lambda_{f}^{N}}{N} \left( \frac{N}{\lambda_{f}^{\left[\frac{N}{2}\right]}} + \frac{N}{\lambda_{f}^{\left[ N^{\frac{1}{6}} \right]}} + 1 + \frac{\left[ N^{\frac{1}{3}} \right]}{N - \left[ N^{\frac{1}{6}} \right]} \right), 
\end{eqnarray*}

yielding an upper bound for $\pi_{S} (f, N)$ for an appropriate $\sigma_{2}$. 
\end{proof} 

\section{Proofs of Theorems \ref{thm:two}, \ref{thm:three} and \ref{thm:four}}
\label{sec:mmd}
In number theory, Mertens' prime counting function is defined as $\mathcal{M}(x) = \displaystyle\sum_{p\, \leq\, x} \frac{1}{p}$, where $p$ is a prime. In \cite{ana:2015, ana:2017}, the authors work with the counting function given by $\mathcal{M}(N) = \displaystyle \sum_{n\, \leq\, N} \frac{O_{T} (n)}{e^{hn}}$, where $O_{T} (n)$ denotes the number of orbits of length $n$ for the Dyck shift and the Motzkin shift respectively and $h$ denotes the topological entropy of the respective shift. In Theorem \ref{thm:two}, we consider a generalisation of this function and have derived a comparability result by considering the orbits of the skew product map. In Theorem \ref{thm:three}, we consider the generalised Meissel's sum $\displaystyle{\sum\limits_{n\, \ge\, 1} \dfrac{C_{S} (f, n)}{n^{k} \lambda_{f}^{n}}}$ for some fixed $k > 0$, even as the authors, in \cite{ana:2015, ana:2017}, consider the sum $\displaystyle \sum_{n\, \ge\, 1} \dfrac{O_{T}(n)}{n^{k} e^{hn}}$ for some fixed $k > 0$. Moreover, we study the Dirichlet series of $C_{S} (f,n)$ in Theorem \ref{thm:four}. With this as the backdrop, we proceed to prove the remainder of the theorems. We begin with the proof of Theorem \ref{thm:two}. 

\begin{proof}[of Theorem \ref{thm:two}] 

We know, from Equation \eqref{eqn:mobinv} that 
\begin{eqnarray} 
\label{ineq:mertlb} 
\sum_{n\, \le\, N} \frac{C_{S} (f, n)}{\lambda_{f}^{n}} & = & \sum_{n\, \le\, N_{0}} \frac{1}{n} \frac{1}{\lambda_{f}^{n}} \sum_{d\; :\; d|n} \mu \left( \frac{n}{d} \right) E_{S} (f, d)\ +\ \sum_{n\, >\, N_{0}}^{N} \frac{1}{n} \frac{1}{\lambda_{f}^{n}} \sum_{d\; :\; d|n} \mu \left( \frac{n}{d} \right) E_{S} (f, d) \nonumber \\ 
& & \hspace{+3cm} (\text{where $N_{0}$ is determined as in the inequality \eqref{ineq:mund}}) \nonumber \\ 
& \ge & \text{some constant, independent of $N$}\ +\ \sum_{n\, >\, N_{0}}^{N} \frac{1}{n} \frac{\kappa_{1}}{2}. 
\end{eqnarray} 


Further, to obtain an upper bound, we split the summation into two parts, as in the proof of Theorem \ref{thm:one}; {\it i.e.}, when $d < n$ and when $d = n$. For $d < n$, we have 
\begin{eqnarray} 
\label{ineq:mertub1} 
\sum_{n\, \le\, N} \frac{1}{n} \frac{1}{\lambda_{f}^{n}} \sum_{d < n\; :\; d|n} \mu \left( \frac{n}{d} \right) E_{S} (f, d) & \le & \sum_{n\, \le\, N} \frac{1}{\lambda_{f}^{n}} c_{1} E_{S} \left( f, \left[ \frac{n}{2} \right] \right) \nonumber \\ 
& \le & c_{1} \kappa_{2} \sum_{n\, \le\, N} \frac{1}{\lambda_{f}^{n - \left[ \frac{n}{2} \right]}} \nonumber \\ 
& = & \text{some constant}. 
\end{eqnarray} 

And for $d = n$, we have 
\begin{equation} 
\label{ineq:mertub2} 
\sum_{n\, \le\, N} \frac{1}{n} \frac{1}{\lambda_{f}^{n}} E_{S} (f, n)\ \ \le\ \ \kappa_{2} \sum_{n\, \le\, N} \frac{1}{n}. 
\end{equation} 

We now state a theorem, as may be found in \cite{hw:2008}, that completes the proof. 

\begin{theorem}[\cite{hw:2008}, Theorem 422, Page 461] 
\label{thm:eulermasch} 
For $x \in \mathbb{R}_{+}$, we have 
\[ \sum_{n\, \le\, x} \frac{1}{n}\ \ =\ \ \log x + \gamma + O \left( \frac{1}{x} \right), \] 
where $\gamma$ denotes the Euler-Mascheroni constant. 
\end{theorem} 

Thus, combining the bounds in the inequalities \eqref{ineq:mertlb}, \eqref{ineq:mertub1}, \eqref{ineq:mertub2} and using Theorem \ref{thm:eulermasch}, we get positive constants $\sigma_{3}$ and $\sigma_{4}$ such that 
\[ \sigma_{3} \log N\ \ \le\ \ \sum_{n\, \le\, N} \frac{C_{S} (f, n)}{\lambda_{f}^{n}}\ \ \le\ \ \sigma_{4} \log N. \] 
\vspace{-9pt} 
\end{proof} 

We now prove Theorem \ref{thm:three}. 

\begin{proof}[of Theorem \ref{thm:three}]
We know from Equation \eqref{eqn:mobinv} that 
\[ C_{S} (f, n)\ =\ \frac{1}{n} \sum_{d\, :\, d|n} \mu \left( \frac{n}{d} \right) E_{S} (f, d)\ \le\ \frac{\kappa_{2}}{n} \sum_{d\, =\, 1}^{n} \lambda_{f}^{d}\ \le\ \frac{\kappa_{2}}{n} \frac{\lambda_{f}^{n + 1}}{\lambda_{f} - 1}. \] 

On the other hand, we know from Equation \eqref{ineq:mund} that 
\[ C_{S} (f, n)\ \ \ge\ \ \frac{\kappa_{1}}{2n} \lambda_{f}^{n}\ \ \ \ \forall n > N_{0}. \] 

Hence, 
\begin{eqnarray*} 
\sum_{n\, \ge\, 1} \frac{C_{S} (f, n)}{n^{k} \lambda_{f}^{n}} & \le & \frac{\kappa_{2} \lambda_{f}}{\lambda_{f} - 1} \sum_{n\, \ge\, 1} \frac{1}{n^{k + 1}}\ \hspace{+2cm}\ \text{and} \\ 
\sum_{n\, \ge\, 1} \frac{C_{S} (f, n)}{n^{k} \lambda_{f}^{n}} & \ge & \sum_{n\, =\, 1}^{N_{0}} \frac{C_{S} (f, n)}{n^{k} \lambda_{f}^{n}} + \frac{\kappa_{1}}{2} \sum_{n\, >\, N_{0}} \frac{1}{n^{k + 1}}. 
\end{eqnarray*} 

Since we know from Riemann integration that 
\[ \frac{1}{k}\ =\ \int_{1}^{\infty} \frac{1}{x^{k + 1}} \mathrm{d}x\ \le\ \sum_{n\, \ge\, 1} \frac{1}{n^{k + 1}}\ \le\ 1 + \int_{1}^{\infty} \frac{1}{x^{k + 1}} \mathrm{d}x\ =\ 1 + \frac{1}{k}, \]
we conclude that exists positive constants, say $\sigma_{5}$ and $\sigma_{6}$ such that 
\[ \frac{\sigma_{5}}{k}\ \ \le\ \ \sum_{n\, \ge\, 1} \frac{C_{S} (f, n)}{n^{k} \lambda_{f}^{n}}\ \ \le\ \ \frac{\sigma_{6}}{k}. \]
\end{proof} 

Finally, we prove Theorem \ref{thm:four} after stating a theorem, as can be found in \cite{hw:2008}. 

\begin{theorem}[\cite{hw:2008}, Section 17.6, Page 328]
\label{thm:dirc} 
Suppose $\left\{ a_{n} \right\}$ and $\left\{ b_{n} \right\}$ are sequences that are related as $a_{n} = \sum\limits_{d\, :\, d|n} b_{d}$, then 
\[ \sum\limits_{n\, \ge\, 1} \dfrac{a_{n}}{n^{z}}\ \ =\ \ \zeta (z) \sum\limits_{n\, \ge\, 1} \dfrac{b_{n}}{n^{z}}, \] 
where $\zeta$ denotes the Riemann zeta function and ${\rm Re}(z) > 1$.  
\end{theorem} 

\begin{proof}[of Theorem \ref{thm:four}] 
Since $E_{S} (f, n) = \sum\limits_{d\, :\, d|n} d C_{S} (f, d)$, Theorem \ref{thm:dirc} implies  
\begin{equation} 
\label{eqn:dirczeta} 
\sum_{n\, \ge\, 1} \frac{E_{S} (f, n)}{n^{z + 1}}\ \ =\ \ \zeta (z + 1) \sum_{n\, \ge\, 1} \frac{n C_{S} (f, n)}{n^{z + 1}}, 
\end{equation} 
for $z$ with ${\rm Re} (z) > 1$. Consider the Dirichlet series of the sequence $C_{S} (f, n)$ given by 
\[ \mathscr{D}_{S}^{f} (z)\ \ =\ \ \sum_{n\, \ge\, 1} \frac{C_{S} (f, n)}{n^{z}}. \] 
Hence, making use of Equation \eqref{eqn:dirczeta}, we have 
\begin{equation} 
\label{ineq:dircub} 
\mathscr{D}_{S}^{f} (z)\ =\ \frac{1}{\zeta (z + 1)} \sum_{n\, \ge\, 1} \frac{E_{S} (f, n)}{n^{z + 1}}\ \le\ \frac{\kappa_{2}}{\zeta (z + 1)} \left( \lambda_{f} + \frac{\lambda_{f}^{2}}{2^{z + 1}} + \frac{\lambda_{f}^{3}}{3^{z + 1}} + \cdots \right). 
\end{equation} 
Also, 
\begin{equation} 
\label{ineq:dirclb} 
\mathscr{D}_{S}^{f} (z)\ \ge\ \frac{\kappa_{1}}{\zeta (z + 1)} \left( \lambda_{f} + \frac{\lambda_{f}^{2}}{2^{z + 1}} + \frac{\lambda_{f}^{3}}{3^{z + 1}} + \cdots \right). 
\end{equation} 
Hence, combining the inequalities \eqref{ineq:dircub} and \eqref{ineq:dirclb}, there exists positive constants $\sigma_{7}$ and $\sigma_{8}$ such that 
\[ \sigma_{7} \sum_{n\, \ge\, 1} \frac{\lambda_{f}^{n}}{n^{z + 1}}\ \ \le\ \ \mathscr{D}_{S}^{f} (z)\ \ =\ \ \sum_{n\, \ge\, 1} \frac{C_{S} (f, n)}{n^{z}}\ \ \le\ \ \sigma_{8} \sum_{n\, \ge\, 1} \frac{\lambda_{f}^{n}}{n^{z + 1}}. \] 
\end{proof} 

\section{Concluding remarks} 
\label{sec:cor}

In this section, we write a few corollaries of the main results that we have proved in this manuscript. Note that when $f \equiv 0,\ C_{S} (0, n)$ counts the number of closed orbits of $S$ of length $n$ while $E_{S} (0, n)$ counts the number of periodic points of period $n$. A simple calculation reveals $E_{S} (0, n) = \left( r_{1} + r_{2} + \cdots + r_{M} \right)^{n} + M^{n}$. 

\begin{corollary} For the function $f \equiv 0$, we have 
\begin{enumerate} 
\item $\pi_{S} (0, N)\ \ =\ \ \sum\limits_{n\, \le\, N}  C_{S} (0, n)\ \ \asymp\ \ \dfrac{\left( r_{1} + r_{2} + \cdots + r_{M} \right)^{N}}{N}$. 
\item $\sum\limits_{n\, \le\, N}  \dfrac{C_{S} (0, n)}{\left( r_{1} + r_{2} + \cdots + r_{M} \right)^{n}}\ \ \asymp\ \ \log N$. 
\item $\sum\limits_{n\, \ge\, 1}  \dfrac{C_{S} (0, n)}{n^{k} \left( r_{1} + r_{2} + \cdots + r_{M} \right)^{n}}\ \ \asymp\ \ \dfrac{1}{k}\ \ \ \ \text{for some fixed}\ \ \ k > 0$. 
\item $\sum\limits_{n\, \ge\, 1} \dfrac{C_{S} (0, n)}{n^{z}}\ \ \asymp\ \ \dfrac{1}{\zeta (z)} \sum\limits_{n\, \ge\, 1} \dfrac{\left( r_{1} + r_{2} + \cdots + r_{M} \right)^{n}}{n^{z + 1}}$. 
\end{enumerate} 
\end{corollary} 

Suppose we only have one map, say $R_{1}$ in our collection of maps, whose degree is given by $r_{1} \ge 2$. Then, the appropriate shift space consists of a single point and thus, can be omitted in our expression. Thus, in this setting, we denote the respective quantities associated to the skew-product map $S$, by merely employing the notation to express the corresponding quantity to the rational map $R_{1}$. In such a dynamical system, we consider the restriction of the map $R_{1}$ in its Julia set, denoted by $J_{R_{1}}$. Let $f$ be some constant function, say $c$ defined on $J_{R_{1}}$. Then, we have the following corollary. 

\begin{corollary} 
\label{cor:singlemap} 
Let $f \equiv c$ be a constant map defined on the Julia set $J_{R_{1}}$ of some rational map $R_{1}$ with degree $r_{1} \ge 2$. Suppose $h(R_{1})$ denotes the entropy of the dynamical system $R_{1}$ restricted on its Julia set. Then, for the quantities $E_{R_{1}} (f, N)$ and $C_{R_{1}} (f, N)$, (analogous to ones defined in Section \eqref{sec:prelim}), that now keeps count of the periodic points in $J_{R_{1}}$, we have 
\begin{enumerate} 
\item $\pi_{R_{1}} (f, N)\ \ =\ \ \sum\limits_{n\, \le\, N}  C_{R_{1}} (f, n)\ \ \asymp\ \ \dfrac{e^{N \left( h(R_{1}) + c \right)}}{N}$. 
\item $\sum\limits_{n\, \le\, N}  \dfrac{C_{R_{1}} (f, n)}{e^{n \left( h(R_{1}) + c \right)}}\ \ \asymp\ \ \log N$. 
\item $\sum\limits_{n\, \ge\, 1}  \dfrac{C_{R_{1}} (f, n)}{n^{k} e^{n \left( h(R_{1}) + c \right)}}\ \ \asymp\ \ \dfrac{1}{k}\ \ \ \ \text{for some fixed}\ \ \ k > 0$. 
\item $\sum\limits_{n\, \ge\, 1} \dfrac{C_{R_{1}} (f, n)}{n^{z}}\ \ \asymp\ \ \dfrac{1}{\zeta (z)} \sum\limits_{n\, \ge\, 1} \dfrac{e^{n \left( h(R_{1}) + c \right)}}{n^{z + 1}}$. 
\end{enumerate} 
\end{corollary} 

We now quote some results that may be helpful in the proof of Corollary \ref{cor:singlemap}. Recall that the Julia set of a rational map $R_{1}$ of degree $r_{1} \ge 2$ is defined as the closure of the set of all repelling periodic points of $R_{1}$, see \cite{ml:1986}. Further, we know from \cite{cf:2006} that the number of repelling periodic points of $R_{1}$ of period $n$ denoted by ${\rm Rep.Per}_{n} (R_{1})$ is bounded by 
\[ r_{1}^{n} - \sum_{d < n\, :\, d|n} r_{1}^{d} - 4n (r_{1} - 1)\ \ \le\ \ \# {\rm Rep.Per}_{n} (R_{1})\ \ \le\ \ 2 r_{1}^{n}. \] 
Also, $h (R_{1}) = \log r_{1}$. 

Finally, we conclude the paper with an explanation for the constant $\lambda_{f}$ that we have written as the standard hypothesis in all the main theorems. We now recall the same. Suppose there exist constants $\lambda_{f} > 1$ and $\kappa_{1},\ \kappa_{2} > 0$ such that $\kappa_{1} \left( \lambda_{f} \right)^{n} \le E_{S} (f, n) \le \kappa_{2} \left( \lambda_{f} \right)^{n}$, where $S$ is the skew-product map defined on $X$ and $f : X \longrightarrow \mathbb{R}$. 

Define a power series 
\[ \rho_{f} (z)\ \ =\ \ \sum_{n\, \ge\, 1} E_{S} (f, n) \frac{z^{n}}{n}. \]
Then, the radius of convergence of the power series $\rho_{f}$ is given by $\dfrac{1}{\lambda_{f}}$. The motivation for such a power series comes from the dynamical zeta function that Parry and Pollicott work with in \cite{pp:1990}. Further, one may observe that, for $|z| < \dfrac{1}{\lambda_{f}}$, we have $\rho_{f} (z) \asymp \log \left( \dfrac{1}{1 - z \lambda_{f}} \right)$. 
\bigskip 
\bigskip 



\bigskip 
\bigskip 

\emph{Authors' contact coordinates:} \\ 

{\bf Subith Gopinathan} \\ 
Indian Institute of Science Education and Research Thiruvananthapuram (IISER-TVM). \\ 
email: \texttt{subith21@iisertvm.ac.in} \\ 
\medskip 

{\bf Bharath Krishna Seshadri} \\ 
Indian Institute of Science Education and Research Thiruvananthapuram (IISER-TVM). \\ 
email: \texttt{bharathmaths21@iisertvm.ac.in} \\ 
\medskip 

{\bf Shrihari Sridharan} \\ 
Indian Institute of Science Education and Research Thiruvananthapuram (IISER-TVM). \\ 
email: \texttt{shrihari@iisertvm.ac.in}

\end{document}